\documentclass[twoside]{amsart}
\usepackage{amsmath,amssymb,amscd,amsthm,amsxtra,amsfonts}
\usepackage{mathtools,mathrsfs,dsfont,xparse}
\usepackage{graphicx,epstopdf}
\usepackage{multirow}
\usepackage{cases}
\usepackage{enumerate}
\usepackage{xcolor}
\usepackage{tikz,pgfplots}
\usetikzlibrary{matrix}
\usepackage{mathtools}
\usepackage[margin=1.1 in]{geometry}
\newtheorem{theorem}{Theorem}[section]
\newtheorem{lemma}[theorem]{Lemma}
\newtheorem{proposition}[theorem]{Proposition}

\newcommand{\R}{\mathbb{R}}

\newcommand{\beq}{\begin{equation}}
\newcommand{\eeq}{\end{equation}}
\newcommand{\beqq}{\begin{equation*}}
\newcommand{\eeqq}{\end{equation*}}

\theoremstyle{definition}

\newtheorem{claim}[theorem]{Claim}
\theoremstyle{remark}
\newtheorem{remark}[theorem]{Remark}

\numberwithin{equation}{section}

\usepackage[colorlinks=true, pdfstartview=FitV, linkcolor=blue, citecolor=blue, urlcolor=blue]{hyperref}

\allowdisplaybreaks[4]

\DeclareDocumentCommand{\abs}{s m}{
  \operatorname{}
  \IfBooleanTF{#1}{#2}{\left|#2\right|}}

\DeclareDocumentCommand{\norm}{s m}{
  \operatorname{}
  \IfBooleanTF{#1}{#2} {\left\| #2\right\|}}

\DeclareDocumentCommand{\inner}{s m}{
  \operatorname{}
  \IfBooleanTF{#1}{#2}{\left \langle#2\right \rangle}}

\DeclareDocumentCommand{\parenthese}{s m}{
  \operatorname{}
  \IfBooleanTF{#1}{#2}{\left(#2\right)}}

\DeclareDocumentCommand{\square}{s m}{
  \operatorname{}
  \IfBooleanTF{#1} {#2}{\left[#2\right]}}

\DeclareDocumentCommand{\bracket}{s m}{
  \operatorname{}
  \IfBooleanTF{#1}{#2}{\left\{#2\right\}}}

\numberwithin{equation}{section}
\mathtoolsset{showonlyrefs}

\begin{document}

\address{Xueying  Yu
\newline \indent Department of Mathematics, University of Washington\indent 
\newline \indent  C138 Padelford Hall Box 354350, Seattle, WA 98195,\indent }
\email{xueyingy@uw.edu}
%\thanks{$^1$  X.Yu  is funded in part by  an AMS-Simons travel grant. }

\address{Haitian Yue
\newline \indent Institute of Mathematical Sciences, ShanghaiTech University\newline\indent
Pudong, Shanghai, China.}
\email{yuehaitian@shanghaitech.edu.cn}
%\thanks{$^2$ H. Yue is supported by a start-up funding of ShanghaiTech University.}

\address{Zehua Zhao
\newline \indent Department of Mathematics and Statistics, Beijing Institute of Technology,
\newline \indent MIIT Key Laboratory of Mathematical Theory and Computation in Information Security,
\newline \indent  Beijing, China. \indent}
\email{zzh@bit.edu.cn}
%\thanks{$^3$ Z. Zhao is supported by NSFC-12101046 and the Beijing Institute of Technology Research Fund Program for Young Scholars.}

\title[On the decay property of the cubic fourth-order Schr\"odinger equation]{On the decay property of the cubic fourth-order Schr\"odinger equation}
\author{Xueying Yu, Haitian Yue and Zehua Zhao}
\maketitle

\begin{abstract}
In this short paper, we prove that the solution of the cubic fourth-order Schr\"odinger equation (4NLS) on $\mathbb{R}^d$ ($5 \leq d \leq 8$) enjoys the same (pointwise) decay property as its linear solution does. This result is proved via a bootstrap argument based on the corresponding global result Pausader \cite{Pau1}. This result can be extended to more general dispersive equations (including some more 4NLS models) with scattering asymptotics.

\bigskip

\noindent \textbf{Keywords}: fourth-order Schr\"odinger equation, dispersive estimate, long time behavior
\bigskip

\noindent \textbf{Mathematics Subject Classification (2020)} Primary: 35Q55; Secondary: 35R01, 37K06, 37L50.

\end{abstract}

\setcounter{tocdepth}{1}
\tableofcontents

\parindent = 10pt     
\parskip = 8pt

\section{Introduction}
\subsection{Statement of main results}
We consider the cubic, defocusing fourth-order Schr\"odinger (4NLS) initial value problem as follows
\begin{equation}\label{maineq}
    (i\partial_t+\Delta^{2}_{\mathbb{R}^d})u = \mu u|u|^2, \quad u(0,x)=u_0(x) \in H^{d+\epsilon}(\mathbb{R}^d),
\end{equation} 
where $\mu=-1$, $5\leq d \leq 8$.
\begin{remark}
 Here we note that the $d+\epsilon$-regularity requirement for the initial data is necessary for our main result due to technical reasons, and if one considers the global existence or scattering behavior, it is natural to consider $H^2$ data which is consistent with the energy space. See \cite{Pau1}.
\end{remark}
If one considers the linear solution to \eqref{maineq} (letting $\mu=0$), the following decay estimate holds, (see \cite{BKS} for more details)
\begin{equation}\label{linear}
   \|u(t)\|_{L_{x}^{\infty} (\R^d)} \lesssim t^{-\frac{d}{4}} \|u_0\|_{L^1_x (\R^d) }. 
\end{equation}
We also note that after interpolating with the mass conservation law, one obtains
\begin{equation}
   \|u(t)\|_{L_{x}^{p} (\R^d)} \lesssim t^{-\frac{d}{4}(1-\frac{2}{p})} \|u_0\|_{L^{\frac{p}{p-1}}_x (\R^d) }, 
\end{equation}
where $p>2$. The goal of this paper is to show the nonlinear solution to \eqref{maineq}  ($\mu=-1$) also satisfies the decay property
\begin{equation}\label{eq 1.2}
   \|u(t)\|_{L_{x}^{\infty}(\R^d)} \lesssim_{\textmd{data}}  t^{-\frac{d}{4}}, 
\end{equation}
where the constant depends on the initial datum. Now we present the main theorem.
\begin{theorem}\label{mainthm}
Consider $u$ satisfies \eqref{maineq} with initial data $u_0 \in L^1_x \cap H^{d+\epsilon}_x (\R^d)$ for any $\epsilon>0$. Then, there exists a constant $C_{u_0}$ depending on the initial profile $u_0$ such that for $t>0$
\begin{equation}
   \|u(t)\|_{L_{x}^{\infty} (\R^d)} \leq C_{u_0} t^{-\frac{d}{4}}. 
\end{equation}
\end{theorem}
\begin{remark}
We again emphasize that the requirement of $(d+\epsilon)$ regularity of initial data is due to technical reasons. See Section \ref{sec main} when we work on the main estimates (especially the estimate of  $F_3$-term). It would very interesting to ask if one could relax the regularity requirement on data. Moreover, it is natural to consider the data in $L^1$ since we discuss the dispersive estimate.
%It is interesting to think about weakening this condition.
\end{remark}
\begin{remark}
One may consider other dispersive models (including more 4NLS models) and try to prove the analogue of Theorem \ref{mainthm}. We will discuss more in Section \ref{sec rmk}.
\end{remark}
The  result above can be regarded as the `\emph{nonlinear decay property}', i.e. the nonlinear solution of a dispersive equation enjoys the same (pointwise) decay property as its linear solution does. Such results are often based on the global existence result of the initial value problem and some assumptions posed on the initial data are needed (higher regularity or finite invariance). We note that \cite{Pau1} has proved that \eqref{maineq} is global well-posed and scattering. Moreover, the scattering norm (a spacetime norm)  $L_{t,x}^{\frac{2(d+4)}{d-4}}$ of the solution $u(t)$ is finite. Heuristically, scattering means the nonlinear solution behaves like a linear solution asymptotically. Thus for a dispersive model with scattering property, it is natural to study if the nonlinear solution decays pointwise like its linear solution. We will explain more motivations and background in the next subsection.

\subsection{Background}
Fourth-order Schr\"odinger (4NLS) equations  have been introduced by Karpman \cite{Karpman} and Karpman and Shagalov \cite{KS} to investigate the stabilizing role of higher-order dispersive effects for the soliton instabilities (e.g. the finite time blowup).
The following work \cite{FIP} by Fibich, Ilan and Papanicolaou studies the  self-focusing (i.e. finite time blowup in dimension two) and singularity formation of such 4NLS equations  from the mathematical viewpoint.
More precursory research works on the basic properties of 4NLS can be found in \cite{BKS,GW2002,guo2016scattering,HHW1, HHW2, Segata}.
It is worth mentioning that the defocusing energy-critical 4NLS in dimension eight was first proved in the series of work by Pausader \cite{Pau2, Pau1}
and then the higher dimension cases ($d\geq 9$) are handled by Miao, Xu and Zhao \cite{MXZ2}.
For more developments on 4NLS, we refer to 
\cite{MWZ,MXZ1, MZ,Pau3,PS,PX,Zheng} in Euclidean spaces.

The 4NLS equations  (similar as nonlinear Schr\"odinger equations) are called dispersive equations because their linear parts are dispersive. The dispersive effect for the linear fourth-order Schr\"odinger equation makes the initial data spread out by the evolution while the mass of the linear solution is conserved, and hence the size of the linear solution decays for large time $t$ as in \eqref{eq 1.2}.
In this paper, we will prove that this dispersive decay also holds for the nonlinear equation. We refer to a recent result \cite{fan2021decay} on the NLS analogue and the references therein for more discussions on the research line of nonlinear decay property for dispersive equations.

The current paper proves this property for a model other than NLS. This is one motivation for us to write this paper. Another motivation is that the authors recently studied the scattering theory for 4NLS on waveguide manifolds where showing decay property is a crucial step. (See Remark A.3. in \cite{yu2021global}.) One can show the decay property in the sense of $\|u(t,x)\|_{L_x^{r}} \rightarrow 0$ as $t \rightarrow \infty$ for some $r>2$ which is sufficient to show the scattering for a subcritical model. Thus we are curious if one can obtain the (stronger) \emph{pointwise} decay estimate as the NLS case. We study the Euclidean case first and we expect the analogous result for the waveguide case,

\subsection{The strategy of the proof}
Our philosophy of studying nonlinear decay property is: if a dispersive model (for example, NLS) is scattering, one can show persistence of regularity and spacetime bounds, i.e. the control of scattering norm (in many cases, people show the finiteness of scattering norm first which implies scattering. See \cite{Taobook}). Then one can prove the nonlinear decay property via the bootstrap argument by controlling the nonlinearity.

Following this idea,  we first recall the scattering result for \eqref{maineq} in \cite{Pau1}. Then we show the persistence of regularity of \eqref{maineq} (see Section \ref{sec overview}). This step is standard (see Lemma 3.2 in \cite{Iteam1} for a NLS analogue). Then we decompose the solution (into the linear solution and the nonlinear evolution  via Duhamel formula) and estimate the nonlinearity in order to apply the bootstrap argument to show the nonlinear decay result, i.e. Theorem \ref{mainthm}. Technically, we need to further decompose the nonlinear part into three terms, i.e. $F_1$, $F_2$ and $F_3$ via splitting the time interval and estimate them respectively (see Section 3 for more details).
\subsection{Organization of the rest of this paper}
In Section \ref{sec overview}, we overview the global dynamics for the model which was investigated in \cite{Pau1} and show the persistence of regularity for \eqref{maineq}; in Section \ref{sec main}, we give the proof for the main theorem, i.e. Theorem \ref{mainthm} via the bootstrap argument; in Section \ref{sec rmk}, we make a few remarks on this research line.
\subsection{Notations}
Throughout this note, we use $C$ to denote  the universal constant and $C$ may change line by line. We say $A\lesssim B$, if $A\leq CB$. We say $A\sim B$ if $A\lesssim B$ and $B\lesssim A$. We also use notation $C_{B}$ to denote a constant depends on $B$. We use the usual $L^{p}$ spaces and $L^2$-based Sobolev spaces $H^{s}$.
\subsection{Acknowledgment}
X. Yu was funded in part by an AMS-Simons travel grant. H. Yue was supported by a start-up funding of ShanghaiTech University. Z. Zhao was supported by the NSF grant of China (No. 12101046) and the Beijing Institute of Technology Research Fund Program for Young Scholars. Part of this work was done while the first two authors were in residence at the Institute for Computational and Experimental Research in Mathematics in Providence, RI, during the {\it Hamiltonian Methods in Dispersive and Wave Evolution Equations} program. 

\section{Overview of the global result}\label{sec overview}

We recall known results on the cubic 4NLS model \eqref{maineq} in \cite{Pau1}. We summarize the main result of \cite{Pau1} (global well-posedness and scattering) by the following proposition.

\begin{proposition}[Theorem 1.1 in \cite{Pau1}]
Let $1 \leq d \leq 8$. For any $u_0 \in H^2 (\R^d)$, there exists a global solution $u \in C(\R , H^2 (\R^d))$ of \eqref{maineq} with initial datum $u(0) = u_0$. If $5 \leq d \leq 8$, the global solution also scatters in $H^2 (\R^d)$. That is, there exist
$f^{\pm} \in H^2(\mathbb{R}^d)$ such that
\begin{equation*}
\lim_{t \rightarrow \pm \infty} \norm{u(t,x)-e^{it\Delta^2}f^{\pm}}_{H^2 (\mathbb{R}^d)} =0.
\end{equation*}
Moreover, the global solution is bounded in the following  spacetime norm 
\begin{align*}
\norm{u}_{L_{t,x}^{\frac{2(d+4)}{d-4}} (\R \times \R^d)} < \infty .
\end{align*}
\end{proposition}
We note that the above result is a little more general (since it concerns a larger range of dimensions) and in this work we only need to focus on the $5\leq d \leq 8$ case when the scattering behavior occurs. Furthermore, we can show persistence of regularity for \eqref{maineq} as follows.

\begin{lemma}[Persistence of regularity]\label{lem PR}
Let $k \in \Bbb N$, $I$ be a compact time interval, and let $u$ be a finite energy solution to \eqref{maineq} on $I \times \R^d$ obeying the bounds
\begin{align*}
\norm{u}_{L_{t,x}^{\frac{2(d+4)}{d-4}} (I \times \R^d)} \leq M.
\end{align*}
Then, if $t_0 \in I$ and $u(t_0) \in H^k (\R^d)$, 
\begin{align}\label{eq Sbdd}
\norm{u}_{\dot{S}^k (I \times \R^d)} \leq C (M , M(u)) \norm{u(t_0)}_{\dot{H}^k (\R^d)}. 
\end{align}
\end{lemma}

\begin{proof}[Proof of Lemma \ref{lem PR}]
We first divide the time interval $I$ into $N$ subintervals  $I_j : = [T_j ,T_{j+1}]$ such that $I = \cup_{j=1}^N I_j $ and on each $I_j$
\begin{align*}
\norm{u}_{L_{t,x}^{\frac{2(d+4)}{d-4}} (I_j \times \R^d)} \leq \delta ,
\end{align*}
where $\delta$ will be chosen later. We have on each $I_j$ by the Strichartz estimates,
\begin{align*}
\norm{u}_{\dot{S}^k (I_j \times \R^d)} & \leq \norm{u(T_j)}_{\dot{H}^k (\R^d)} + C  \norm{\nabla^k (\abs{u}^2 u)}_{L_t^1 L_x^2 (I_j \times \R^d)}\\
& \leq \norm{u(T_j)}_{\dot{H}^k (\R^d)}  + C \norm{\nabla^k u}_{\dot{S}^0 (I_j \times \R^d)}  \norm{u}_{L_{t,x}^{\frac{2(d+4)}{d-4}} (I_j \times \R^d)}^{\frac{d-4}{2}} \norm{u}_{\dot{S}^0 (I_j \times \R^d)}^{4-\frac{d}{2}}\\
& \leq \norm{u(T_j)}_{\dot{H}^k (\R^d)}  +  C \delta^{\frac{d-4}{2}} \norm{\nabla^k u}_{\dot{S}^0 (I_j \times \R^d)}   \norm{u}_{\dot{S}^0 (I_j \times \R^d)}^{4-\frac{d}{2}} ,
\end{align*}
where the constant $C$ might vary from line to line.

Choosing $\delta$ small enough, we obtain the bound
\begin{align}\label{eq SSbdd}
\norm{u}_{\dot{S}^k (I_j \times \R^d)} \leq 2 C \norm{u(T_j)}_{\dot{H}^k (\R^d)} . 
\end{align}
Then the bound \eqref{eq Sbdd} follows by adding up the bounds \eqref{eq SSbdd} we have on each subinterval.
\end{proof}

\section{Proof of the main theorem}\label{sec main}
In this section, we give the proof for Theorem \ref{mainthm}. We will adapt the scheme as in \cite{fan2021decay} with suitable modifications for the 4NLS case. One main difference is that. 

We define 
\begin{equation}
A(\tau):=\sup_{0 \leq s\leq \tau} s^{\frac{d}{4}}\|u(s)\|_{L_{x}^{\infty} (\R^d)}.
\end{equation}
Note that $A(\tau)$ is monotone increasing.
We intend to prove that there exists some constant, depending on $u_{0}$, so that 
\begin{equation}
A(\tau)\leq C_{u_{0}}, \quad \text{ for any } \tau \geq 0.
\end{equation}
Recall we have persistence of regularity, thus for any given $l$, one can find $C_{l}$ so that,
\begin{equation}
A(\tau)\leq C_{l}, \quad \text{ for any }  0\leq \tau\leq l,
\end{equation} 
and the solution is continuous  in time in $L^{\infty}$ since we are working on high regularity data.

Thus, Theorem \ref{mainthm} follows from the following bootstrap lemma.
\begin{lemma}\label{lem: bootstrap}
There exists a constant $C_{u_{0}}$, such that if one has $A(\tau)\leq C_{u_{0}}$, then for $\tau \geq 0$, one can improve the bound by $A(\tau)\leq \frac{C_{u_{0}}}{2}$  for any $\tau 
\geq 0$.
\end{lemma}

\begin{proof}[Proof of Lemma \ref{lem: bootstrap}]

%Now we turn to the proof of Lemma \ref{lem: bootstrap}. 
One important mission in this lemma is to choose a suitable $C_{u_0}$, and we will see the reason of choice of $C_{u_0}$ in the proof later. 
%The way to choose $C_{u_{0}}$ will become clear in the proof. 
For fixed $\tau$, we only need to prove that for any $t\leq \tau$, one has
\begin{equation}\label{eq: goal}
\|u(t)\|_{L_{x}^{\infty} (\R^d)}\leq \frac{C_{u_{0}}}{2}t^{-\frac{d}{4}}.
\end{equation}
We recall here, by bootstrap assumption, we can apply the following estimates in the proof
\begin{equation}\label{eq: boosassu}
\|u(t)\|_{L_{x}^{\infty} (\R^d)}\leq C_{u_{0}}t^{-\frac{d}{4}}.
\end{equation}
 Observe, for any $\delta$, we can choose $L$, so that for one has 
\begin{equation}\label{eq: scatteringdecay}
\big( \int_{L/2}^{\infty}\|u\|^{\frac{2(d+4)}{d-4}}_{L_{x}^{\frac{2(d+4)}{d-4}} (\R^d)}\, dt \big)^{\frac{d-4}{2(d+4)}} \leq \delta.
\end{equation}

We will first fix two special constants $\delta, L$ in the proof,  though the exact   choices of these two parameters will only be made clear later.

We will only study $t\geq L$, and estimate all $t\leq L$ directly via
\begin{equation}
\|u(t)\|_{L_{x}^{\infty}}\leq A(L)t^{-3/2}, \quad t\leq L.
\end{equation}

Next, by Duhamel's Formula, we write the nonlinear solution $u(t,x)$ as follows,
\begin{equation}
    u(t,x)=e^{it\Delta^2}u_0+i\int_0^t e^{i(t-s)\Delta^2}(|u|^2u)(s) \, ds : =u_l+u_{nl}.
\end{equation}
Clearly dispersive estimates give for some constant $C_{0}$,
\begin{equation}\label{eq ul}
\|u_{l}(t)\|_{L_{x}^{\infty} (\R^d)}\leq C_{0}t^{-\frac{d}{4}}\|u_{0}\|_{L_{x}^{1} (\R^d)}.
\end{equation} 

Then, we split $u_{nl}$ into 
\begin{equation}
u_{nl}=F_{1}+F_{2}+F_{3}
\end{equation}
where
\begin{equation}
\begin{aligned}
    &F_1(t)=i\int_0^M e^{i(t-s)\Delta^2}(|u|^2u)(s) \, ds,\\
    &F_2(t)=i\int_M^{t-M} e^{i(t-s)\Delta^2}(|u|^2u)(s) \, ds,\\
    &F_3(t)=i\int_{t-M}^t e^{i(t-s)\Delta^2}(|u|^2u)(s) \, ds.\\
    \end{aligned}
\end{equation}

\begin{claim}\label{claim}
We claim that 
\begin{enumerate}
\item
$\|F_{1}(t)\|_{L_{x}^{\infty}(\R^d)} \lesssim MM_{1}^{3}t^{-\frac{d}{4}}$;
\item
$\|F_{2}(t)\|_{L^{\infty}_x (\R^d)} \leq \frac{1}{6}C_{u_{0}}t^{-\frac{d}{4}}$;
\item
$\|F_{3}(t)\|_{L_{x}^{\infty}(\R^d)}\leq \frac{1}{6}C_{u_{0}}t^{-\frac{d}{4}}$.
\end{enumerate}

\end{claim}

Assuming Claim \ref{claim}, let us continue the bootstrap argument. 

For all $t\leq \tau$, assuming $A(\tau)\leq C_{u_{0}}$, we derive
\begin{itemize}
\item For $t\leq L$, we have 
\begin{equation}
u(t)\leq A(L)t^{-\frac{d}{4}}.
\end{equation}
\item For $L\leq t\leq \tau$, by Claim \ref{claim}, we write 
\begin{equation}
u(t)\leq \parenthese{C(\|u_{0}\|_{L_{x}^{1}(\R^d)}+MM_{1}^{3})+\frac{1}{6}C_{u_{0}}+\frac{1}{6}C_{u_{0}}}t^{-\frac{d}{4}}.
\end{equation}
\end{itemize}

Thus, if one chooses
\begin{equation}
C_{u_{0}}:=2A(L)+3C(\|u_{0}\|_{L_{x}^{1} (\R^d)}+MM_{1}^{3}),
\end{equation}
then the desired estimates follows, which completes the bootstrap argument in Lemma \ref{lem: bootstrap}.

Now we are left to prove Claim \ref{claim}. 
\begin{proof}[Proof of Claim \ref{claim}]
For {\it (1)},  we estimate $F_{1}$ by 
\begin{equation}\label{eq: esf1}
\begin{aligned}
\|F_{1}(t)\|_{L_{x}^{\infty} (\R^d)}&\leq \int_{0}^{M}\|e^{i(t-s)\Delta^2}|u|^{2}u(s)\|_{L_{x}^{\infty} (\R^d) }\, ds\\
&\lesssim M(t-M)^{-\frac{d}{4}}\sup_{s}\|u^{3}(s)\|_{L_{x}^{1}(\R^d)}\\
&\lesssim Mt^{-\frac{d}{4}}\sup_{s}\|u(s)\|_{H_{x}^{3}(\R^d)}^{3}\\
&\lesssim MM_{1}^{3}t^{-\frac{d}{4}}.
\end{aligned}
\end{equation}
We note that we will choose $M$ satisfying $M<\frac{t}{2}$, then we can bound $(t-M)^{-\frac{d}{4}}$ by $t^{-\frac{d}{4}}$ (multiplying a constant), which has been used in the estimates above. 

For {\it (2)}, we first recall \eqref{eq: scatteringdecay}. Then We consider a pointwise estimate applying the bootstrap assumption, together with the H\"older and the Sobolev inequality as follows 
\begin{equation}\label{eq: pt3}
\|e^{i(t-s)\Delta^2}|u(s)|^{2}u(s)\|_{L_{x}^{\infty}(\R^d)}\lesssim (t-s)^{-\frac{d}{4}}\|u(s)\|^2_{H_x^{2}(\R^d)}  \|u(s)\|_{L_{x}^{\infty}(\R^d)}\lesssim C_{u_{0}}M^2_{1}(t-s)^{-\frac{d}{4}}s^{-\frac{d}{4}}.
\end{equation}
And we can estimate $F_{2}$ via 
\begin{equation}\label{eq: esf2pre3}
\aligned
\|F_{2}(t)\|_{L^{\infty}_x(\R^d)} &\leq CC_{u_{0}}M^2_{1}\int_{M}^{t-M}(t-s)^{-\frac{d}{4}}s^{-\frac{d}{4}}  \, ds\\
&\leq CC_{u_{0}}M^2_{1}\int_{M}^{\frac{t}{2}}(t-s)^{-\frac{d}{4}}s^{-\frac{d}{4}} \,  ds\\
&+CC_{u_{0}}M^2_{1}\int_{\frac{t}{2}}^{t-M}(t-s)^{-1}s^{-\frac{d}{4}} \,  ds \\
&\leq 2CC_{u_{0}}M^2_{1} t^{-\frac{d}{4}}\int_{M}^{\frac{t}{2}}s^{-\frac{d}{4}} \,  ds \\
&+ 2CC_{u_{0}}M^2_{1} t^{-\frac{d}{4}} \int_{\frac{t}{2}}^{t-M}(t-s)^{-\frac{d}{4}} \,  ds. 
\endaligned
\end{equation}
Choosing $M$, so that
\begin{equation}\label{eq: choiceofM3}
4CC_{u_{0}}M^2_{1} t^{-\frac{d}{4}} \big(\int_{M}^{\infty}s^{-\frac{d}{4}} \, ds\big) \leq \frac{1}{6}C_{u_{0}}t^{-\frac{d}{4}}.
\end{equation}
We note that $\frac{d}{4}>1$ since $5 \leq d\leq 8$. Thus eventually we can estimate $F_{2}$ as
\begin{equation}
\|F_{2}(t)\|_{L^{\infty}_x(\R^d)} \leq \frac{1}{6}C_{u_{0}}t^{-\frac{d}{4}}.
\end{equation}

We now turn to the estimate for {\it (3)} (i.e. estimating $F_{3}$). We first state the following lemma,
\begin{lemma}\label{lem: ele2}
 Let $f(x)$ be an $H_x^{d+\epsilon}$ function in $\mathbb{R}^{d}$, with  
 \begin{equation}
 \|f\|_{L_x^{2}(\R^d)}\leq a, \|f\|_{H_x^{d+\epsilon}(\R^d)}\leq b. 
 \end{equation}
 Then one has 
 \begin{equation}
\|f\|_{L_x^{\infty}(\R^d)}\leq  a^{\frac{d+2\epsilon}{2(d+\epsilon)}}b^{\frac{d}{2(d+\epsilon)}}.
 \end{equation}
 \end{lemma}

\begin{proof}[Proof of Lemma \ref{lem: ele2}]
This follows from Gagliardo-Nirenberg inequality (see \cite{gagliardo1959ulteriori,nirenberg2011elliptic}) as below.
\begin{equation}
    \|f\|_{L_x^{\infty}(\R^d)}\lesssim \|  f\|^{\frac{d+2\epsilon}{2(d+\epsilon)}}_{L_x^{2}(\R^d)} \cdot \| |\nabla|^{d+\epsilon} f\|^{\frac{d}{2(d+\epsilon)}}_{L_x^{2}(\R^d)}.
\end{equation}
\end{proof}
\begin{remark}
As shown below, to deal with $F_3$-term using our method, we need assume at least $(d+\epsilon)$-regularity ($\epsilon>0$ is arbitrarily small). A nature interesting question to ask is if one could relax the regularity assumption. For the NLS case (see \cite{fan2021decay}), there are also higher regularity assumptions. It is also interesting to think if one can lower the regularity assumption.
\end{remark}
\begin{remark}
One may compare this part with the NLS case \cite{fan2021decay}. Unlike the NLS case, using $L^{2}$- and $W^{1,\infty}$-norms to control $L^{\infty}$-norm is not sufficient since the spacial dimension is higher ($5 \leq d \leq 8$).
\end{remark}
Following Lemma \ref{lem: ele2}, we estimate the  $L_x^{2}$- and $H_x^{d+\epsilon}$-norms of $F_{3}$.
Note that $H_x^{d+\epsilon}$ is a Banach algebra under pointwise multiplication (since $d+\epsilon>\frac{d}{2}$), and $e^{i(t-s)\Delta^2}$ is unitary in $H_x^{d+\epsilon}$, we directly estimate $\|F_{3}(t)\|_{H_x^{d+\epsilon}}$ as 
\begin{equation}\label{eq: H3}
\|F_{3}(t)\|_{H_x^{d+\epsilon} (\R^d)}\leq MM_{1}^{3}.
\end{equation}

For $ \|F_{3}(t)\|_{L^{2}_{x}}$, we will use the fact $t-M\geq L/2$ and rely on the scattering decay assumption. Also note $t-M\sim t$ since $t\geq L\geq 100M$. Using the H\"older inequality, we estimate as 
\begin{equation}
\begin{aligned}
\big\|\int_{t-M}^t e^{i(t-s)\Delta^2}|u|^2u \, ds \big\|_{L^{2}_x(\R^d)} &\leq \int_{t-M}^t \| |u|^2u \|_{L_x^2(\R^d)} \, ds \\
  &\leq  \int_{t-M}^t \| u \|^{\frac{(4-\epsilon)d+12\epsilon}{4(d+2\epsilon)}}_{L_x^{2}(\R^d)} \| u \|^{\frac{\epsilon d+4\epsilon}{4(d+2\epsilon)}}_{L_x^{\frac{2(d+4)}{d-4}}(\R^d)}\cdot \| u \|^{\frac{2(d+\epsilon)}{d+2\epsilon}}_{L_x^{\infty}(\R^d)} \,  ds \\ 
    &\leq CM_1^{\frac{(4-\epsilon)d+12\epsilon}{4(d+2\epsilon)}}(C_{u_{0}}t^{-\frac{d}{4}})^{\frac{2(d+\epsilon)}{d+2\epsilon}} \int_{t-M}^t \|u \|^{\frac{\epsilon d+4\epsilon}{4(d+2\epsilon)}}_{L_x^{\frac{2(d+4)}{d-4}}(\R^d)} \, ds \\
     &\leq CM_1^{\frac{(4-\epsilon)d+12\epsilon}{4(d+2\epsilon)}}(C_{u_{0}}t^{-\frac{d}{4}})^{\frac{2(d+\epsilon)}{d+2\epsilon}} \cdot \| u \|^{\frac{\epsilon d+4\epsilon}{4(d+2\epsilon)}}_{L^{\frac{2(d+4)}{d-4}}_t L_x^{\frac{2(d+4)}{d-4}} ([t-M, t] \times \R^d) }  \cdot M^{1-\frac{(d-4)(\epsilon d+4\epsilon)}{8(d+4)(d+2\epsilon)}} \\
          &\leq CM_1^{\frac{(4-\epsilon)d+12\epsilon}{4(d+2\epsilon)}}(C_{u_{0}}t^{-\frac{d}{4}})^{\frac{2(d+\epsilon)}{d+2\epsilon}} \delta^{\frac{\epsilon d+4\epsilon}{4(d+2\epsilon)}} M^{1-\frac{(d-4)(\epsilon d+4\epsilon)}{8(d+4)(d+2\epsilon)}}. \\
\end{aligned}
\end{equation}

Thus, via Lemma \ref{lem: ele2}, we derive
\begin{equation}\label{eq: esf3pre}
\aligned
\big\|\int_{t-M}^t e^{i(t-s)\Delta}|u|^2u \, ds \big\|_{L^{\infty}_x(\R^d)} &\leq \big( CM_1^{\frac{(4-\epsilon)d+12\epsilon}{4(d+2\epsilon)}}(C_{u_{0}}t^{-\frac{d}{4}})^{\frac{2(d+\epsilon)}{d+2\epsilon}} \delta^{\frac{\epsilon d+4\epsilon}{4(d+2\epsilon)}} M^{1-\frac{(d-4)(\epsilon d+4\epsilon)}{8(d+4)(d+2\epsilon)}} )^{\frac{d+2\epsilon}{2(d+\epsilon)}} \cdot  (MM_1^3)^{\frac{d}{2(d+\epsilon)}} \\
&\leq C^{\frac{d+2}{2(d+1)}}C_{u_0}\delta^{\frac{\epsilon d+4\epsilon}{4(d+2\epsilon)}} M^{\alpha(d)} M_1^{\frac{(16-\epsilon)d+12\epsilon}{8(d+\epsilon)}} t^{-\frac{d}{4}},
\endaligned
\end{equation}
where $\alpha(d)=1-\frac{(d-4)(\epsilon d+4\epsilon)}{16(d+4)(d+\epsilon)}$. Thus, by choosing $\delta$ small enough, according to $M, M_{1}$, we can ensure
\begin{equation}\label{eq: esf3}
\|F_{3}(t)\|_{L_{x}^{\infty}(\R^d)}\leq \frac{1}{6}C_{u_{0}}t^{-\frac{d}{4}}.
\end{equation}
We note that we choose $L$, depending on $\delta$, so that the above estimate holds .

Here, we mention that, the choice of $M,L$ does not depend on $C_{u_{0}}$. We will choose $C_{u_{0}}$ depending on $M, L$.
\begin{remark}
In addition to Remark 3.3, we explain why we need to assume at least $d+\epsilon$-regularity for estimating $F_3$-term, First, we need the smallness from the finite scattering norm to close the bootstrap argument. As for the cubic nonlinearity (we put it in $L^2$-norm), we can at most spare $2-$ `number' for $L_x^{\infty}$-norm which is for the decay. (We can not do $2$, if so, the left norm would be $L^2$-norm which does not give us any smallness.) Thus in Lemma \ref{lem: ele2}, we need to assume $d+\epsilon$-regularity to make sure the exponent for $\|f\|_{L^2_x(\R^d)}$ is $2+$ then `$2+$'$\times$ `$2-$' $=1$. 
\end{remark} 
Now we finish the proof of Claim \ref{claim}.
\end{proof}

This ends the proof of Lemma \ref{lem: bootstrap}, and   the main theorem follows from Lemma \ref{lem: bootstrap}.

\end{proof}
\section{Further remarks}\label{sec rmk}

In this section, we make a few remarks on this research line. As mentioned in the introduction, our nonlinear decay result for \eqref{maineq} is based on the corresponding scattering result \cite{Pau1}. We believe that this method can be applied for some other dispersive models with proper modifications. We list a few more models (with scattering behavior) and leave the proofs for interested readers. (We note that one may require more regular initial data in the following models.) 

\begin{enumerate}
\item
\begin{equation}
    (i\partial_t+\Delta^{2}_{\mathbb{R}^d})u = -u|u|^{\frac{8}{d}}, \quad u(0,x)=u_0(x) \in L^2(\mathbb{R}^d),
\end{equation} 
where $d \geq 5$. (See \cite{PS})
 
\item
\begin{equation}
    (i\partial_t+\Delta^{2}_{\mathbb{R}^d})u = - u|u|^{\frac{8}{d-4}}, \quad u(0,x)=u_0(x) \in H^2(\mathbb{R}^d),
\end{equation} 
where $d \geq 9$. (See \cite{MXZ2}) 

\item
\begin{equation}
    (i\partial_t+\Delta^{2}_{\mathbb{R}^d})u = \mu u|u|^{d}, \quad u(0,x)=u_0(x) \in H^2(\mathbb{R}^d),
\end{equation} 
where $1\leq d \leq 4, p>1+\frac{8}{d}$. (See \cite{PX}) 

\end{enumerate}

For the focusing case, this method can be applied as well once we have the global spacetime bound since estimating the nonlinearity and using bootstrap argument have nothing to do with the sign of nonlinearity. It is essentially a \emph{perturbative} method. See \cite{guo2016scattering,MXZ1,PS} for examples. We conclude this paper with one more paragraph on discussing more dispersive models as below.  

For dispersive models rather than 4NLS on Euclidean spaces (with scattering behavior), similar method may be applied to obtain the nonlinear decay property (i.e. the nonlinear solutions enjoy the same pointwise decay property as the linear solutions), such as, higher order (more than four) NLS, 4NLS on waveguide manifolds (see \cite{yu2021global} for a recent result), NLS on waveguide manifolds (see \cite{HP,IPRT3,Z1} for example), NLS with partial harmonic potentials (see \cite{antonelli2015scattering,cheng2021scattering,hani2016asymptotic}), resonant system (see \cite{CGZ,yang2018global}), nonlinear wave equations (see \cite{Taobook}),  Klein–Gordon equation (see \cite{Taobook}). We did not list them explicitly.

\bibliography{4NLS}
\bibliographystyle{plain}
\end{document}